\documentclass[12pt,reqno]{amsart}
\usepackage{amsmath, amsfonts, amssymb, amsthm, hyperref,cite, amsaddr}
\usepackage{bm}
\allowdisplaybreaks[4]
\def\l{\left}
\def\r{\right}
\def\bg{\bigg}
\def\({\bg(}
\def\){\bg)}
\def\t{\text}
\def\f{\frac}

\def\eq{\equiv}

\def\Z{\mathbb Z}

\def\N{\mathbb N}

\def\1{{\bf 1}}

\theoremstyle{plain}
\newtheorem{theorem}{Theorem}[section]
\newtheorem{lemma}{Lemma}
\newtheorem{corollary}{Corollary}
\newtheorem{conjecture}{Conjecture}

\theoremstyle{definition}

\newtheorem*{Acks}{Acknowledgments}
\theoremstyle{remark}

\def\<{\langle}
\def\>{\rangle}

\begin{document}
\hbox{}
\medskip

\title[A parametric congruence motivated by Orr's identity]{A parametric congruence motivated by Orr's identity}

\author{Chen Wang$^{1,*}$ and Zhi-Wei Sun$^2$}
\email{cwang@smail.nju.edu.cn, zwsun@nju.edu.cn}
\address{\small $^1$Department of Applied Mathematics, Nanjing Forestry
University, Nanjing 210037, People's Republic of China\\
$^2$Department of Mathematics, Nanjing
University, Nanjing 210093, People's Republic of China}
\thanks{$^{*}$Corresponding author}

\begin{abstract} For any $m,n\in\mathbb{N}=\{0,1,2\ldots\}$, the truncated hypergeometric series ${}_{m+1}F_m$ is defined by
$$
{}_{m+1}F_m\bigg[\begin{matrix}x_0&x_1&\ldots&x_m\\ &y_1&\ldots&y_m\end{matrix}\bigg|z\bigg]_n=\sum_{k=0}^n\frac{(x_0)_k(x_1)_k\cdots(x_m)_k}{(y_1)_k\cdots(y_m)_k}\cdot\frac{z^k}{k!},
$$
where $(x)_k=x(x+1)\cdots(x+k-1)$ is the Pochhammer symbol. Let $p$ be an odd prime. For $\alpha,z\in\mathbb{Z}_p$ with $\langle -\alpha\rangle_p\equiv0\pmod{2}$, where $\langle x\rangle_p$ denotes the least nonnegative residue of $x$ modulo $p$ for any $x\in\mathbb{Z}_p$, we mainly prove the following congruence motivated by Orr's identity:
$$
{}_2F_1\bigg[\begin{matrix}\frac12\alpha&\frac32-\frac12\alpha\\ &1\end{matrix}\bigg|z\bigg]_{p-1}{}_2F_1\bigg[\begin{matrix}\frac12\alpha&\frac12-\frac12\alpha\\ &1\end{matrix}\bigg|z\bigg]_{p-1}\equiv{}_3F_2\bigg[\begin{matrix}\alpha&2-\alpha&\frac12\\ &1&1\end{matrix}\bigg|z\bigg]_{p-1}\pmod{p^2}.
$$
As a corollary, for any positive integer $b$ with $p\equiv\pm1\pmod{b}$ and $\langle -1/b\rangle_p\equiv0\pmod{2}$, we deduce that
$$
\sum_{k=0}^{p-1}(b^2k+b-1)\frac{\binom{2k}{k}}{4^k}\binom{-1/b}{k}\binom{1/b-1}{k}\equiv0\pmod{p^2}.
$$
This confirms a conjectural congruence of the second author.\\

\noindent{\bf Competing Interests and Funding}: The authors are supported by the National Natural Science Foundation of China (grants 12201301 and 11971222, respectively).

\end{abstract}
\maketitle

\section{Introduction}
\setcounter{lemma}{0}
\setcounter{theorem}{0}
\setcounter{equation}{0}
\setcounter{conjecture}{0}
\setcounter{corollary}{0}
\setcounter{remark}{0}

For any $m,n\in\mathbb{N}=\{0,1,2\ldots\}$, the truncated hypergeometric series ${}_{m+1}F_m$ is defined by
$$
{}_{m+1}F_m\bigg[\begin{matrix}x_0&x_1&\ldots&x_m\\ &y_1&\ldots&y_m\end{matrix}\bigg|z\bigg]_n=\sum_{k=0}^n\frac{(x_0)_k(x_1)_k\cdots(x_m)_k}{(y_1)_k\cdots(y_m)_k}\cdot\frac{z^k}{k!},
$$
where $(x)_k=x(x+1)\cdots(x+k-1)$ is the Pochhammer symbol. Clearly, the above truncated hypergeometric series is the truncation of the original hypergeometric series after the $z^n$ term.

Summation and transformation formulas for hypergeometric series play an important role in the study of the congruence properties of truncated hypergeometric series (see, e.g., \cite{Guo2017,Liu,MaoPan,PTW,Sun2011,Sunthreebinomial,Sun2013,Tauraso, WangPan,WangSun2020,WangXia}). Recall the well-known Clausen's identity (cf. \cite[p. 116]{AAR99})
\begin{equation}\label{clausen}
\l({}_2F_1\bigg[\begin{matrix}\f12\alpha&\f12\beta\\&\f12+\f12(\alpha+\beta)\end{matrix}\bigg|z\bigg]\r)^2={}_3F_2\bigg[\begin{matrix}\alpha&\beta&\f12(\alpha+\beta)\\&\alpha+\beta&\f12+\f12(\alpha+\beta)\end{matrix}\bigg|z\bigg].
\end{equation}
Letting $\beta=1-\alpha$ in \eqref{clausen} we obtain
\begin{equation}\label{clausen'}
\l({}_2F_1\bigg[\begin{matrix}\f12\alpha&\f12-\f12\alpha\\&1\end{matrix}\bigg|z\bigg]\r)^2={}_3F_2\bigg[\begin{matrix}\alpha&1-\alpha&\f12\\&1&1\end{matrix}\bigg|z\bigg].
\end{equation}
Let $p$ be an odd prime and let $\Z_p$ denote the ring of all $p$-adic integers. For any $x\in\Z_p$ let $\<x\>_p$ denote the least nonnegative residue of $x$ modulo $p$. Mao and Pan \cite{MaoPan} proved the following parametric congruence with respect to the identity \eqref{clausen'}:
\begin{equation}\label{clausen'analogue}
\l({}_2F_1\bigg[\begin{matrix}\f12\alpha&\f12-\f12\alpha\\&1\end{matrix}\bigg|z\bigg]_{p-1}\r)^2\eq{}_3F_2\bigg[\begin{matrix}\alpha&1-\alpha&\f12\\&1&1\end{matrix}\bigg|z\bigg]_{p-1}\pmod{p^2},
\end{equation}
where $\alpha,z\in\Z_p$ and $\<-\alpha\>_p$ is even.

When he discussed the differential equation satisfied by the product of two hypergeometric series, Orr discovered the following formula (cf. \cite[p. 180]{AAR99}) similar to \eqref{clausen}:
\begin{equation}\label{orr}
{}_2F_1\bigg[\begin{matrix}\f12\alpha&\f12\beta\\ &\f12(\alpha+\beta)-\f12\end{matrix}\bigg|z\bigg]{}_2F_1\bigg[\begin{matrix}\f12\alpha&\f12\beta-1\\ &\f12(\alpha+\beta)-\f12\end{matrix}\bigg|z\bigg]={}_3F_2\bigg[\begin{matrix}\alpha&\beta-1&\f12(\alpha+\beta)-1\\ &\alpha+\beta-2&\f12(\alpha+\beta)-\f12\end{matrix}\bigg|z\bigg].
\end{equation}
Putting $\beta=3-\alpha$ in \eqref{orr} we have
\begin{equation}\label{orr'}
{}_2F_1\bigg[\begin{matrix}\f12\alpha&\f32-\f12\alpha\\ &1\end{matrix}\bigg|z\bigg]{}_2F_1\bigg[\begin{matrix}\f12\alpha&\f12-\f12\alpha\\ &1\end{matrix}\bigg|z\bigg]={}_3F_2\bigg[\begin{matrix}\alpha&2-\alpha&\f12\\ &1&1\end{matrix}\bigg|z\bigg].
\end{equation}

The main purpose of this paper is to establish a parametric congruence corresponding to \eqref{orr'}.

\begin{theorem}\label{orranalogue}
Let $p$ be an odd prime. Then, for $\alpha,z\in\Z_p$ with $\<-\alpha\>_p\eq0\pmod{2}$ we have
\begin{equation}\label{orranalogueeq}
{}_2F_1\bigg[\begin{matrix}\frac12\alpha&\frac32-\frac12\alpha\\ &1\end{matrix}\bigg|z\bigg]_{p-1}{}_2F_1\bigg[\begin{matrix}\frac12\alpha&\frac12-\frac12\alpha\\ &1\end{matrix}\bigg|z\bigg]_{p-1}\equiv{}_3F_2\bigg[\begin{matrix}\alpha&2-\alpha&\frac12\\ &1&1\end{matrix}\bigg|z\bigg]_{p-1}\pmod{p^2}.
\end{equation}
\end{theorem}

Recently, the second author \cite[Conjecture 19]{Sun2019} posed the following conjecture.

\begin{conjecture}\label{sunconj}
Let $b,n\in\Z^{+}$ and let $p$ be a prime with $p\eq\pm1\pmod{b}$ and $\<-1/b\>_p\eq0\pmod{2}$. Then
\begin{equation}\label{sunconjeq}
\f{1}{n^2\binom{-1/b}{n}\binom{1/b-1}{n}}\sum_{k=0}^{pn-1}(b^2k+b-1)\f{\binom{2k}{k}}{4^k}\binom{-1/b}{k}\binom{1/b-1}{k}\eq0\pmod{p^2}.
\end{equation}
\end{conjecture}
Note that Conjecture \ref{sunconj} with $n=1$ and $b\in\{2,3,4,6\}$ was first stated by the second author in \cite[Conjecture 5.9]{Sun2011}. Our following result confirms Conjecture \ref{sunconj} for $n=1$.

\begin{corollary}\label{orranaloguecor}
Let $b\in\Z^{+}$, and let $p$ be a prime with $p\eq\pm1\pmod{b}$ and $\<-1/b\>_p\eq0\pmod{2}$. Then
\begin{equation}\label{orranaloguecoreq}
\sum_{k=0}^{p-1}(b^2k+b-1)\frac{\binom{2k}{k}}{4^k}\binom{-1/b}{k}\binom{1/b-1}{k}\equiv0\pmod{p^2}.
\end{equation}
\end{corollary}

The relation between Theorem \ref{orranalogue} and Corollary \ref{orranaloguecor} becomes more evident when we write \eqref{sunconjeq} as a difference of truncated hypergeometric series. Note that $$(x)_k/(1)_k=\binom{-x}{k}(-1)^k,\ (1/2)_k/(1)_k=\binom{2k}{k}/4^k\ \ \t{and}\ \ \f{(1+\f1b)_k}{(\f1b)_k}=bk+1.$$
Thus we have
\begin{align}\label{sunconjkey}
&\sum_{k=0}^{p-1}(b^2k+b-1)\frac{\binom{2k}{k}}{4^k}\binom{-1/b}{k}\binom{1/b-1}{k}\notag\\
=\ &b\,{}_3F_2\bigg[\begin{matrix}1+\f1b&1-\f1b&\f12\\&1&1\end{matrix}\bigg|1\bigg]_{p-1}-{}_3F_2\bigg[\begin{matrix}\f1b&1-\f1b&\f12\\&1&1\end{matrix}\bigg|1\bigg]_{p-1}.
\end{align}
In view of \eqref{clausen'analogue} and Theorem \ref{orranalogue}, in order to show Corollary \ref{orranaloguecor}, it suffices to evaluate some truncated ${}_2F_1$ series modulo $p^2$.

Taking $b=2,3,4,6$ in Corollary \ref{orranaloguecor} and noting that
\begin{gather*}
\binom{-1/2}{k}=\f{\binom{2k}{k}}{(-4)^k},\ \binom{-1/3}{k}\binom{-2/3}{k}=\f{\binom{2k}{k}\binom{3k}{k}}{27^k},\\
\binom{-1/4}{k}\binom{-3/4}{k}=\f{\binom{2k}{k}\binom{4k}{2k}}{64^k},\ \binom{-1/6}{k}\binom{-5/6}{k}=\f{\binom{6k}{3k}\binom{3k}{k}}{432^k},
\end{gather*}
we obtain the following congruences which confirm \cite[Conjecture 5.9]{Sun2011} with $a=1$.
\begin{corollary}\label{specialcases}
Let $p$ be an odd prime. If $p\eq1\pmod{3}$, then
\begin{equation}\label{specialcaseseq1}
\sum_{k=0}^{p-1}\f{9k+2}{108^k}\binom{2k}{k}^2\binom{3k}{k}\eq0\pmod{p^2}.
\end{equation}
If $p\eq1,3\pmod{8}$, then
\begin{equation}\label{specialcaseseq2}
\sum_{k=0}^{p-1}\f{16k+3}{256^k}\binom{2k}{k}^2\binom{4k}{2k}\eq0\pmod{p^{2+\delta_{p,3}}},
\end{equation}
where $\delta_{p,q}$ denotes the Kronecker delta symbol which takes $1$ or $0$ according as $p=q$ or not. \ If $p\eq1\pmod{4}$, then
\begin{equation}\label{specialcaseseq3}
\sum_{k=0}^{p-1}\f{4k+1}{64^k}\binom{2k}{k}^3\eq0\pmod{p^2}
\end{equation}
and
\begin{equation}\label{specialcaseseq4}
\sum_{k=0}^{p-1}\f{36k+5}{12^{3k}}\binom{6k}{3k}\binom{3k}{k}\binom{2k}{k}\eq0\pmod{p^{2+\delta_{p,5}}}.
\end{equation}
\end{corollary}

In 2019, the authors \cite{WangSun2019} established the modulus $p^3$ congruence for
$$
{}_4F_3\bigg[\begin{matrix}\alpha&1+\f12\alpha&\alpha&\alpha\\&\f12\alpha&1&1\end{matrix}\bigg|1\bigg]_{p-1}
$$
which is a parametric extension of \eqref{specialcaseseq3}, where $p$ is an odd prime and $\alpha$ is a $p$-adic unit.

We shall prove Theorem \ref{orranalogue} and Corollary \ref{orranaloguecor} in the next section.

\section{Proofs of Theorem \ref{orranalogue} and Corollary \ref{orranaloguecor}}
\setcounter{lemma}{0}
\setcounter{theorem}{0}
\setcounter{equation}{0}
\setcounter{conjecture}{0}
\setcounter{corollary}{0}
\setcounter{remark}{0}

To show Theorem \ref{orranalogue} we need the following result due to Tauraso \cite[Theorem 2]{Tauraso2016}.

\begin{lemma}\label{tauraso}
For any prime $p>3$ and $p$-adic integer $x$ we have
\begin{equation}\label{taurasoeq}
\l(\sum_{k=1}^{p-1}\binom{2k}{k}x^k\r)\l(\sum_{k=1}^{p-1}\binom{2k}{k}\f{x^k}k\r)\eq2\sum_{k=1}^{p-1}\binom{2k}{k}(H_{2k-1}-H_k)x^k\pmod{p},
\end{equation}
where $H_k=\sum_{j=1}^k1/j$ is the $k$th harmonic number.
\end{lemma}

\noindent{\it Proof of Theorem \ref{orranalogue}}. Throughout the proof, we always set $a=\<-\alpha\>_p$.

{\it Case 1}. $a\leq p-3$.

Let
\begin{align*}
\Phi(x)=&\ {}_2F_1\bigg[\begin{matrix}\f12(-a+x)&\f12(a+3-x)\\&1\end{matrix}\bigg|z\bigg]_{p-1}{}_2F_1\bigg[\begin{matrix}\f12(-a+x)&\f12(a+1-x)\\&1\end{matrix}\bigg|z\bigg]_{p-1}\\
&-{}_3F_2\bigg[\begin{matrix}-a+x&a+2-x&\f12\\&1&1\end{matrix}\bigg|z\bigg]_{p-1},
\end{align*}
where $x\in\Z_p$. Expanding $\Phi(x)$ we have
\begin{align*}
\Phi(x)=&\sum_{k=0}^{p-1}\sum_{l=0}^{p-1}\f{(\f12(-a+x))_k(\f12(a+3-x))_k(\f12(-a+x))_l(\f12(a+1-x))_l}{k!^2l!^2}\cdot z^{k+l}\\
&-\sum_{j=0}^{p-1}\f{(-a+x)_j(a+2-x)_j(\f12)_j}{j!^3}\cdot z^j.
\end{align*}
Clearly, $\Phi(x)$ is a rational function in $x$ and $\Phi(0)\in\Z_p$. Therefore, by \cite[Lemma 4.1]{PTW} we have
\begin{equation}\label{phitp}
\Phi(tp)\eq\Phi(0)+tp\Phi'(0)\pmod{p^2}
\end{equation}
for any $t\in\Z_p$. In particular, we have
\begin{equation}\label{phip}
\Phi(p)\eq\Phi(0)+p\Phi'(0)\pmod{p^2}.
\end{equation}
Note that $-a/2,\ -a\in\{1-p,\ldots,-1,0\}$. In view of \eqref{orr'}, we have
$$
\Phi(0)={}_2F_1\bigg[\begin{matrix}-\f12a&\f12(a+3)\\&1\end{matrix}\bigg|z\bigg]{}_2F_1\bigg[\begin{matrix}-\f12a&\f12(a+1)\\&1\end{matrix}\bigg|z\bigg]-{}_3F_2\bigg[\begin{matrix}-a&a+2&\f12\\&1&1\end{matrix}\bigg|z\bigg]=0.
$$
Since $a\leq p-3$ we have $(a+3-p)/2,\ (a+1-p)/2,\ a+2-p\in\{1-p,\ldots,-1,0\}$. By \eqref{orr'} we also have
\begin{align*}
\Phi(p)=&\ {}_2F_1\bigg[\begin{matrix}\f12(-a+p)&\f12(a+3-p)\\&1\end{matrix}\bigg|z\bigg]{}_2F_1\bigg[\begin{matrix}\f12(-a+p)&\f12(a+1-p)\\&1\end{matrix}\bigg|z\bigg]\\
&-{}_3F_2\bigg[\begin{matrix}-a+p&a+2-p&\f12\\&1&1\end{matrix}\bigg|z\bigg]\\
=&\ 0.
\end{align*}
Substituting these into \eqref{phitp} and \eqref{phip}, we get
$$
\Phi(tp)\eq\Phi(0)=0\pmod{p^2}.
$$
Putting $t=(a+\alpha)/p$ we immediately obtain the desired result.

\medskip

{\it Case 2}. $a=p-1$.

We can directly verify this case for $p=3$. Below we assume $p>3$. In this case, we may write $\alpha=1+pt$ for some $t\in\Z_p$ since $\alpha\eq1\pmod{p}$. Then we have
\begin{align}\label{2F13/2}
&{}_2F_1\bigg[\begin{matrix}\f12\alpha&\f32-\f12\alpha\\&1\end{matrix}\bigg|z\bigg]_{p-1}=\sum_{k=0}^{p-1}\f{(\f12+\f12pt)_k(1-\f12pt)_k}{k!^2}\cdot z^k\notag\\
\eq&\sum_{k=0}^{p-1}\f{(\f12)_k}{k!}\l(1+pt\sum_{j=0}^{k-1}\f{1}{2j+1}-\f12ptH_k\r)z^k\notag\\
\eq&\sum_{k=0}^{p-1}\binom{2k}{k}(1+ptH_{2k}-ptH_k)\l(\f{z}4\r)^k\pmod{p^2}.
\end{align}
Similarly, we also have
\begin{equation}\label{2F11/2}
{}_2F_1\bigg[\begin{matrix}\f12\alpha&\f12-\f12\alpha\\&1\end{matrix}\bigg|z\bigg]_{p-1}\eq1-\f12pt\sum_{k=1}^{p-1}\f{\binom{2k}{k}}{k}\l(\f{z}4\r)^k\pmod{p^2}
\end{equation}
and
\begin{equation}\label{3F21/2}
{}_3F_2\bigg[\begin{matrix}\alpha&2-\alpha&\f12\\&1&1\end{matrix}\bigg|z\bigg]_{p-1}\eq\sum_{k=0}^{p-1}\binom{2k}{k}\l(\f{z}4\r)^k\pmod{p^2}.
\end{equation}
Combining \eqref{2F13/2}--\eqref{3F21/2} with Lemma \ref{tauraso}, we arrive at
\begin{align*}
&{}_2F_1\bigg[\begin{matrix}\f12\alpha&\f32-\f12\alpha\\&1\end{matrix}\bigg|z\bigg]_{p-1}{}_2F_1\bigg[\begin{matrix}\f12\alpha&\f12-\f12\alpha\\&1\end{matrix}\bigg|z\bigg]_{p-1}-{}_3F_2\bigg[\begin{matrix}\alpha&2-\alpha&\f12\\&1&1\end{matrix}\bigg|z\bigg]_{p-1}\\
\eq&-\f12pt\l(\sum_{k=0}^{p-1}\binom{2k}{k}\l(\f{z}4\r)^k\r)\l(\sum_{k=1}^{p-1}\f{\binom{2k}{k}}{k}\l(\f{z}{4}\r)^k\r)+pt\sum_{k=0}^{p-1}\binom{2k}{k}(H_{2k}-H_k)\l(\f{z}4\r)^k\\
\eq&-\f12pt\sum_{k=1}^{p-1}\f{\binom{2k}{k}}{k}\l(\f{z}{4}\r)^k-pt\sum_{k=1}^{p-1}\binom{2k}{k}(H_{2k-1}-H_k)\l(\f{z}{4}\r)^k+pt\sum_{k=0}^{p-1}\binom{2k}{k}(H_{2k}-H_k)\l(\f{z}4\r)^k\\
=&\ 0\pmod{p^2}.
\end{align*}

The proof of Theorem \ref{orranalogue} is now complete.\qed

\medskip

Let us recall the definition and the main properties of the $p$-adic Gamma function introduced by Morita \cite{Morita} as a $p$-adic analogue of the classical Gamma function. Let $p$ be an odd prime. For any $n\in\N$ define
$$
\Gamma_p(n)=(-1)^n\prod_{\substack{1\leq k<n\\ p\nmid k}}k.
$$
In particular, set $\Gamma_p(0)=1$. Clearly, the values of $\Gamma_p(n)$ belong to the group $\Z_p^{\times}$ of $p$-adic units. It is known that the definition of $\Gamma_p(n)$ can be extended to $\Z_p$ since $\N$ is a dense subset of $\Z_p$ in the sense of $p$-adic norm $|\cdot|_p$. That is, for all $x\in\Z_p$ we can define
$$
\Gamma_p(x)=\lim_{\substack{n\in\N\\|x-n|_p\to0}}\Gamma_p(n).
$$
Similar to the classical Gamma function, the $p$-adic Gamma function has some interesting properties. For example, for any $x\in\Z_p$ we have
\begin{equation}\label{padicgamma1}
\f{\Gamma_p(x+1)}{\Gamma_p(x)}=\begin{cases}-x\quad&\t{if}\ p\nmid x,\\
-1\quad&\t{if}\ p\mid x,
\end{cases}
\end{equation}
and
\begin{equation}\label{padicgamma2}
\Gamma_p(x)\Gamma_p(1-x)=(-1)^{p-\<-x\>_p}.
\end{equation}
The reader may consult \cite{Robert00} for more properties of the $p$-adic Gamma function.

\begin{lemma}[Mao and Pan {\cite[Theorem 1.1]{MaoPan}}]\label{maopanlem} Let $p$ be an odd prime and $\alpha,\beta\in\Z_p$. If $\<-\alpha\>_p+\<-\beta\>_p<p$, then
\begin{equation}\label{maopanlemeq1}
{}_2F_1\bigg[\begin{matrix}\alpha&\beta\\&1\end{matrix}\bigg|1\bigg]_{p-1}\eq-\f{\Gamma_p(1-\alpha-\beta)}{\Gamma_p(1-\alpha)\Gamma_p(1-\beta)}\pmod{p^2}.
\end{equation}
If $\<-\alpha\>_p+\<-\beta\>_p\geq p$, then
\begin{equation}\label{maopanlemeq2}
{}_2F_1\bigg[\begin{matrix}\alpha&\beta\\&1\end{matrix}\bigg|1\bigg]_{p-1}\eq(\alpha+\beta+\<-\alpha\>_p+\<-\beta\>_p-p)\f{\Gamma_p(1-\alpha-\beta)}{\Gamma_p(1-\alpha)\Gamma_p(1-\beta)}\pmod{p^2}.
\end{equation}
\end{lemma}

\begin{lemma}\label{keylem}
Let $b\in\{2,3,4,\ldots\}$, and let $p$ be a prime with $p\eq\pm1\pmod{b}$ and $\<-1/b\>_p\eq0\pmod{2}$. Then
\begin{equation}\label{keylemeq}
b\,{}_3F_2\bigg[\begin{matrix}1+\f1b&1-\f1b&\f12\\&1&1\end{matrix}\bigg|1\bigg]_{p-1}\eq{}_3F_2\bigg[\begin{matrix}\f1b&1-\f1b&\f12\\&1&1\end{matrix}\bigg|1\bigg]_{p-1}\pmod{p^2}.
\end{equation}
\end{lemma}

\begin{proof}
It is clear that $p$ is odd and
$$
\l\<\f1b-1\r\>_p=p-1-\l\<-\f1b\r\>_p\eq0\pmod{2}.
$$
Therefore, by \eqref{clausen'analogue} and Theorem \ref{orranalogue} we have
$$
{}_3F_2\bigg[\begin{matrix}\f1b&1-\f1b&\f12\\&1&1\end{matrix}\bigg|1\bigg]_{p-1}\eq{}_2F_1\bigg[\begin{matrix}\f1{2b}&\f12-\f1{2b}\\&1\end{matrix}\bigg|1\bigg]_{p-1}^2\pmod{p^2}
$$
and
$$
{}_3F_2\bigg[\begin{matrix}1+\f1b&1-\f1b&\f12\\&1&1\end{matrix}\bigg|1\bigg]_{p-1}\eq{}_2F_1\bigg[\begin{matrix}1+\f1{2b}&\f12-\f1{2b}\\&1\end{matrix}\bigg|1\bigg]_{p-1}{}_2F_1\bigg[\begin{matrix}\f1{2b}&\f12-\f1{2b}\\&1\end{matrix}\bigg|1\bigg]_{p-1}\pmod{p^2}.
$$

Now we assume $p\eq1\pmod{b}$. Note that
$$\l\<-\f1{2b}\r\>_p+\l\<\f1{2b}-\f12\r\>_p=\f{p-1}{2b}+\f{p-1}{2}-\f{p-1}{2b}<p$$
and
$$\l\<-1-\f1{2b}\r\>_p+\l\<\f1{2b}-\f12\r\>_p=\f{p-1}{2b}-1+\f{p-1}{2}-\f{p-1}{2b}<p.$$
Thus, by Lemma \ref{maopanlem} we deduce that
\begin{align*}
&b\,{}_3F_2\bigg[\begin{matrix}1+\f1b&1-\f1b&\f12\\&1&1\end{matrix}\bigg|1\bigg]_{p-1}-{}_3F_2\bigg[\begin{matrix}\f1b&1-\f1b&\f12\\&1&1\end{matrix}\bigg|1\bigg]_{p-1}\\
\eq&b\cdot\f{\Gamma_p(-\f12)\Gamma_p(\f12)}{\Gamma_p(\f12+\f1{2b})^2\Gamma_p(-\f1{2b})\Gamma_p(1-\f1{2b})}-\f{\Gamma_p(\f12)^2}{\Gamma_p(1-\f1{2b})^2\Gamma_p(\f12+\f{1}{2b})^2}\\
=&0\pmod{p^2},
\end{align*}
where we have used the facts
$$
\Gamma_p\l(\f12\r)=\f12\Gamma_p\l(-\f12\r)\ \t{and}\ \Gamma_p\l(1-\f1{2b}\r)=\f1{2b}\Gamma_p\l(-\f1{2b}\r).
$$

Below we suppose that $p\eq-1\pmod{b}$. Note that
$$\l\<-\f1{2b}\r\>_p+\l\<\f1{2b}-\f12\r\>_p=\f12\l(p-\f{p+1}b\r)+\f{p-1}{2}-\f12\l(p-\f{p+1}b\r)<p$$
and
$$\l\<-1-\f1{2b}\r\>_p+\l\<\f1{2b}-\f12\r\>_p=\f12\l(p-\f{p+1}b\r)-1+\f{p-1}{2}-\f12\l(p-\f{p+1}b\r)<p.$$
Similarly, as in the case $p\eq1\pmod{b}$, we also have
$$
b\, {}_3F_2\bigg[\begin{matrix}1+\f1b&1-\f1b&\f12\\&1&1\end{matrix}\bigg|1\bigg]_{p-1}-{}_3F_2\bigg[\begin{matrix}\f1b&1-\f1b&\f12\\&1&1\end{matrix}\bigg|1\bigg]_{p-1}\eq0\pmod{p^2}.
$$

This completes the proof.
\end{proof}

\medskip

\noindent{\it Proof of Corollary \ref{orranaloguecor}}. Clearly, \eqref{orranaloguecoreq} holds for $b=1$. If $b\geq2$, then the desired result easily follows from Lemma \ref{keylem}.\qed

\begin{Acks}
The authors would like to thank the two referees for helpful comments.
\end{Acks}

\end{document}